\documentclass[graybox]{svmult}

\usepackage{mathptmx}       
\usepackage{helvet}         
\usepackage{courier}        
\usepackage{type1cm}        

\usepackage{makeidx}         
\usepackage{graphicx}        
\usepackage{multicol}        
\usepackage[bottom]{footmisc}

\usepackage{amsmath}
\usepackage{graphicx}
\usepackage{amssymb}
\usepackage{amsfonts}
\usepackage{graphicx,color}
\usepackage{amsbsy}

\makeindex             

\newcommand{\A}{\mathcal{A}_{p}}

\newcommand{\ZZ}{\Z_{p^r}}

\newcommand{\F}{{\mathbb F}}
\newcommand{\Z}{{\mathbb Z}}

\newcommand{\beq}{\begin{equation}}
\newcommand{\eeq}{\end{equation}}
\newcommand{\bmat}{\left[ \begin{array}}
\newcommand{\emat}{\end{array} \right]}

\newcommand{\C}{{\mathcal{C}}}

\newcommand{\R}{{\mathcal{R}}}

\begin{document}

\title{The dual of convolutional codes over $\mathbb Z_{p^r}$}
\author{Mohammed El Oued and Diego Napp and Raquel Pinto and Marisa Toste}
\institute{Mohammed El Oued \at  FSMMath Department, University of Monastir, Monastir 5050, Tunisia \email{wadyel@yahoo.fr}
\and Diego Napp \at  CIDMA - Center for Research and Development in Mathematics and Applications, Department of Mathematics, University of Aveiro, Aveiro, Portugal \email{diego@ua.pt}
\and Raquel Pinto \at  CIDMA - Center for Research and Development in Mathematics and Applications, Department of Mathematics, University of Aveiro, Aveiro, Portugal \email{raquel@ua.pt}
\and Marisa Toste \at  CIDMA - Center for Research and Development in Mathematics and Applications, Instituto Politécnico de Coimbra-ESTGOH, Coimbra, Portugal \email{marisatoste@gmail.com}}
%
%
%
\maketitle

\abstract{ An important class of codes widely used in applications is the class of convolutional codes. Most of the literature of convolutional codes is devoted to convolutional codes over finite fields. The extension of the concept of convolutional codes from finite fields to finite rings have attracted much attention in recent years due to fact that they are the most appropriate codes for phase modulation. However convolutional codes over finite rings are more involved and not fully understood. Many results and features that are well-known for convolutional codes over finite fields have not been fully investigated in the context of finite rings. In this paper we focus in one of these unexplored areas, namely, we investigate the dual codes of convolutional codes over finite rings. In particular we study the $p$-dimension of the dual code of a convolutional code over a finite ring. This contribution can be considered a generalization and an extension, to the rings case, of the work done by Forney and McEliece on the dimension of the dual code of a convolutional code over a finite field. }

\section{Introduction}
\label{sec:1}
Codes play an important role in our days. They are implemented in most of all communications systems in order to detect and correct errors that can be introduced during the transmission of information. Convolutional codes over finite rings were first introduced by \cite{ma89} and have becoming more relevant for communication systems that combine coding and modulation.

We will consider convolutional codes constituted by left compact sequences in $\Z_{p^r}$, where $p$ is a prime and $r$ an integer greater than $1$, i.e., the codewords of the code will be of the form
$$
\begin{array}{cccc}
w: & \mathbb Z & \rightarrow &  \Z_{p^r}^n \\
& t & \mapsto & w_t
\end{array}
$$
where $w_t=0$ for $t < k$ for some $k \in \mathbb Z$. These sequences can be represented by Laurent series, $w(D)= \displaystyle \sum_{t=k}^\infty w_t D^t$. Let us denote by $\Z_{p^r}((D))$ the ring of Laurent series over $\mathbb Z_{p^r}$. Moreover, we will represent the ring of polynomials over $\Z_{p^r}$ by $\Z_{p^r}[D]$ and the ring of rational matrices over $\Z_{p^r}$ by $\Z_{p^r}(D)$. More precisely, $\Z_{p^r}(D)$ is the set
$$
\small
\{\frac{p(D)}{q(D)}: p(D),q(D) \in \Z_{p^r}[D] \mbox{ and the coefficient of the smallest power of $D$ in $q(D)$ is a unit}\}
$$
modulo the equivalence relation
$$
\frac{p(D)}{q(D)} \sim \frac{p_1(D)}{q_1(D)} \mbox{ if and only if } p(D)q_1(D) = f_1(D) q(D).
$$
Convolutional codes over finite rings behave very differently from convolutional codes over finite fields due to the existence of zero divisors. One main difference is that a convolutional code over a finite field $\mathbb F$ is always a free module over $\F((D))$ which does not happen in the ring case. In order to deal with this problem we will consider a new type of basis, for $\Z_{p^r}[D]$-submodules of $\Z_{p^r}^n[D]$, which will allow us to define a kind of basis for every convolutional code, called $p$-basis, and a related type of dimension, called $p$-dimension. This notions have been extensively used in the last decade \cite{fa01,jo99,ko95,ku09,ku07,NappPintoToste2015,Mo15,el13}, extending the ideas of $p$-adic expansion, $p$-dimension, $p$-basis, etc, used in the context of $\Z_{p^r}$-submodules of $\Z_{p^r}^n$, \cite{ca00a,no00,no01,va96}.

In this paper we will study the dual of a convolutional code over $\Z_{p^r}[D]$. In particular, we will show that the dual of a convolutional code is also a convolutional code and we will relate the $p$-dimensions of a convolutional code and its dual. Hence this work generalizes the results derived by Forney and McEliece \cite{fo70,mc98} of convolutional codes over finite fields.

\section{The module $\mathbb Z_{p^r}^n[D]$}
\label{sec:2}

Any element in $\Z_{p^r}^n$ can be written uniquely as a linear combination of $1,p,p^2,\dots $ $\dots, p^{r-1}$, with coefficients in $\A=\{0,1, \dots,p-1\} \subset \Z_{p^r}$ (called the $p$-adic expansion of the element) \cite{ca00a}. Note that all elements of $\A \backslash\{0\}$ are units. This property provides a kind of linear independence on the elements of $\A$. In \cite{va96}, the authors considered this property to define a special type of linear combination of vectors, called {\it $p$-linear combination}, which allowed to define the notion of {\it $p$-generator sequence}, {\it $p$-basis} and {\it $p$-dimension} for every submodule of $\Z^n_{p^r}$. These notions were extended for vectors in \cite{ku07} and we recall them in this section.

\begin{definition}\cite{ku07} Let $v_1(D), \dots, v_k(D) $ be in $\Z^n_{p^r}[D]$. The vector $\displaystyle \sum_{j=1}^k a_j(D) v_j(D)$,
with $a_j(D) \in \A [D]$, is said to be a {\bf $\boldsymbol{p}$-linear combination} of ${v_1(D), \dots, v_k(D)}$ and the set of all $p$-linear combination of ${v_1(D), \dots, v_k(D)}$ is called the {\bf $\boldsymbol{p}$-span} of $\{v_1(D), \dots, v_k(D) \}$, denoted by
$p$-span $(v_1(D), \dots, v_k(D))$.
\end{definition}

Note that the $p$-span of a set of vectors is not always a module. We need to introduce an extra condition to be fulfilled by the vectors.

\begin{definition}\cite{ku07}
An ordered set of vectors $(v_1(D), \dots, v_k(D))$ in $\Z^n_{p^r}[D]$ is said to be a {\bf $\boldsymbol{p}$-generator sequence} if
$p \, v_i(D)$ is a $p$-linear combination of $v_{i+1}(D), \dots, v_k(D)$, $\; \; i=1, \dots, k-1$,
and $p \, v_k(D)=0$.
\end{definition}

\begin{lemma}\cite{ku07}
If $(v_1(D), \dots, v_k(D))$ is a $p$-generator sequence in $\Z^n_{p^r}[D]$ it holds that
$
p\mbox{-span}(v_1(D), \dots, v_k(D))=\mbox{span}(v_1(D), \dots, v_k(D)),
$
and consequently we have that $p$-span$(v_1(D), \dots, v_k(D))$ is a $\ZZ$-submodule of $\Z^n_{p^r}[D]$.
\end{lemma}

Note that if $M=span(v_1(D), \dots, v_k(D))$ is a submodule of $\mathbb Z_{p^r}[D]$, then
\begin{equation}
\label{eq:07}
\begin{split}
(v_1(D), pv_1(D) \dots, & p^{r-1}v_1(D),v_2(D), pv_2(D), \dots, \\
& \dots, p^{r-1}v_2(D), \dots, v_l(D), pv_k(D) \dots, p^{r-1}v_k(D)).
\end{split}
\end{equation}

is a $p$-generator sequence of $M$..

\begin{definition}\cite{ku07}
The vectors $v_1(D), \dots, v_k(D)$  in $\Z_{p^r}^n[D]$ are said to be {\bf $\boldsymbol{p}$-linearly independent} if the only $p$-linear combination of
$v_1(D), \dots, v_k(D)$ that is equal to $0$ is the trivial one.
\end{definition}

\begin{definition}\cite{ku07} An ordered set of vectors $(v_1(D), \dots, v_k(D))$ which is a $p$-linearly independent $p$-generator sequence of a submodule $M$ of $\mathbb Z_{p^r}^n[D]$ is said to be a {\bf $\boldsymbol{p}$-basis} of $M$.
\end{definition}

It is proved in \cite{ku09} that two $p$-bases of a $\Z_{p^r}$-submodule $M$ of $\Z_{p^r}^n[D]$ have the same number of elements. This number of elements is called {\bf $\boldsymbol{p}$-dimension} of $M$ and is denoted by $p$-$dim (M)$.

The same notions and results are satisfied for the module $\mathbb Z_{p^r}^n$ in \cite{va96}. In fact, as mentioned before, these notions were first introduced in this paper for such modules and later extended for the module $\mathbb Z_{p^r}^n[D]$ in \cite{ku07}.


\section{Convolutional Codes over $\Z_{p^r}$}
\label{subsec:2_2}

\begin{definition}
A {\bf convolutional code $\mathcal{C}$} of length $n$ is a $\Z_{p^r}((D))$-submodule of  $\Z^n_{p^r}((D))$ for which there exists a polynomial matrix $ G(D) \in \Z_{p^r}^{\tilde k \times n}[D]$ such that
\begin{eqnarray*}
{\mathcal C} & = & \mbox{Im}_{\Z_{p^r}((D))}  G(D)\\
& = & \left\{u(D) G(D) \in \Z^n_{p^r}((D)) :\, u(D) \in \Z^{\widetilde k}_p((D))\right\}.
\end{eqnarray*}

The matrix $  G(D)$ is called a {\bf generator matrix} of $\C$. If $ G(D)$ is full row rank then it is called an {\bf encoder} of $\C$.

Moreover, if
\begin{eqnarray*}
{\mathcal C}  & = & \mbox{Im}_{{\cal A}_{p^r}((D))} G(D)\\
  & = & \left\{u(D)G(D) \in \Z^n_{p^r}((D)) :\, u(D) \in \mathcal{A}^k_p((D))\right\}.
\end{eqnarray*}
with $G(D) \in \Z_{p^r}^{k \times n}[D]$ a polynomial matrix whose rows form a $p$-basis of $\mathcal{C}$, then $G(D)$ is a {\bf $\boldsymbol{p}$-encoder} of $\mathcal{C}$ and we say that $\C$ has $p$-dimension $k$.

If there exists a constant matrix $G$ such that
$$
{\mathcal C}  = \left\{u(D) G \in \Z^n_{p^r}((D)) :\, u(D) \in \Z^{\widetilde k}_p((D))\right\},
$$
then $\cal C$ is called a {\bf block code}.
\end{definition}

Obviously, block codes are a particular case of convolutional codes. Every block code $\cal C$ admits a generator matrix in standard form \cite{no01}
\begin{equation}\label{standard form}
 G = \left[\begin{array}{ccccccc}
    I_{k_0} & A^0_{1,0} & A^0_{2,0} & A^0_{3,0} & \cdots & A^0_{r-1,0} & A^0_{r,0} \\
        0 & pI_{k_1} & pA^1_{2,1} & pA^1_{3,1} & \cdots & pA^1_{r-1,1} & pA^1_{r,1} \\
       0 & 0 & p^2I_{k_2} & p^2A^2_{3,2} & \cdots & p^2A^2_{r-1,2} & p^2A^2_{r,2} \\
       0 & 0 & 0 & p^{r-1}I_{k_3} & \cdots & 0 & p^{r-1}A^3_{r,3} \\
    \vdots&\vdots&\vdots&\vdots&\ddots&\vdots&\vdots\\
    0 & 0 & 0 & 0 & \cdots  & p^{r-1}I_{k_{r-1}} & p^{r-1}A^{r-1}_{r,r-1}
  \end{array}\right].
\end{equation}
The integers $k_0, k_1, \dots, k_{r-1}$ are called the {\bf parameters} of $ G$. All encoders of $\cal C$ in standard form have the same parameters $k_0, k_1, \dots, k_{r-1}$.


\bigskip

\noindent Note that if $G(D)$ is a generator matrix of a convolutional code $\C$ and $X(D)$ is an invertible rational matrix such that $X(D) G(D)$ is polynomial, then $\mbox{Im}_{\Z_{p^r}((D))} G(D) = \mbox{Im}_{\Z_{p^r}((D))} X(D) G(D)$, which means that $X(D) G(D)$ is also a generator matrix of $\C$.
Thus, the next straightforward result follows. We include its proof for the sake of completeness.
\begin{lemma} \label{rational encoder}Let $\C$ be a submodule of $\Z_{p^r}^n((D))$ given by $\C=\mbox{Im}_{\Z_{p^r}((D))}N(D)$, where $N(D) \in \Z_{p^r}^{\tilde k \times n}(D)$. Then $\C$ is a convolutional code, and if $N(D)$ is full row rank, $\C$ is a free code of dimension $k$.
\end{lemma}

\begin{proof}
Write $N(D)=\left[\frac{p_{ij}(D)}{q_{ij}(D)}\right] $, where $p_{ij}(D), q_{ij}(D) \in \Z_{p^r}[D]$, and the coefficient of the smallest power of $D$ in $q_{ij}(D)$ is a unit. Consider the diagonal matrix $Y(D) \in \Z_{p^r}^{\tilde k \times \tilde k}[D]$ whose element of the row $i$ is the least common multiple of $q_{i1}(D), q_{i2}(D), \dots, q_{i\tilde k}(D)$. Thus $Y(D)$ is invertible and $N(D)=Y(D)^{-1}X(D)$ for some polynomial matrix $X(D) \in \Z_{p^r}^{\tilde k \times n}[D]$. Then $\mbox{Im}_{\Z_{p^r}((D))}N(D) = \mbox{Im}_{\Z_{p^r}((D))}X(D)$, which means that $X(D)$ is a generator matrix of $\C$. The last statement of the lemma follows from the fact that $N(D)$ is full row rank if and only if $X(D)$ is full row rank.
\qed
\end{proof}

\bigskip
\noindent Next we will consider a decomposition of a convolutional code into simpler components. For that we need the following lemma.
\begin{lemma}
Let $M$ be a submodule of  $\Z_{p^r}^n((D))$. Then, there exists a unique family $M_0,\ldots,M_{r-1}$ of free submodules of $\Z_{p^r}^n((D))$ such that
\begin{equation}\label{e1}
M=M_0\oplus pM_1\oplus\ldots\oplus p^{r-1}M_{r-1}.
\end{equation}
\end{lemma}

\begin{proof}

Let $\overline{M}$ be the projection of $M$ over $\Z_p((D))$ and denote its dimension  by $k_0(M)$. Let $M_0$ be the free code over $\ZZ((D))$ of rank $k_0$ satisfying  $\overline{M}=\overline{M_0}$ and $M_0\subset M$.  As $\ZZ^n((D))$ is a semisimple module, $M_0$ admits a complement  code $M_0'$ in $M$. Necessarily, there exists a code $M_1'$ such that $M_0'=pM_1'$.  We have $M=M_0\oplus pM_1'$.
Then by induction we have the result.
\qed
\end{proof}

\noindent Note that if $\cal C$ is a block code, this decomposition as ${\cal C}= {\cal C}_0 \oplus p {\cal C}_1\oplus\ldots\oplus p^{r-1} {\cal C}_{r-1}$ where ${\cal C}_0, \dots , {\cal C}_{r-1}$ are free block codes, is directly derived from a generator matrix in standard form. In fact, if $G$, of the form (\ref{standard form}), is a generator matrix of $\C$ then $p^i {\cal C}_i = \mbox{Im}_{\Z_{p^r}((D))} p^i G_i$, where $G_i=[0 \cdots 0 \; I_{k_i} \; A^i_{2,i} \cdots  A^i_{r,i}]$, $i=0, \dots, r-1$.

\bigskip

\noindent Next we will show that any convolutional code $\C$ also admits such decomposition.

\noindent Let $G(D)$ be a generator matrix of $\C$. If $G(D)$ is full row rank then $\C$ is free and $\C=\C_0$.

\noindent Let us assume now that $G(D)$ is not full row rank. Then the projection of $G(D)$ into $\Z_p[D]$, $\overline{G (D)} \in \Z_p^{k \times n}[D]$, is also not full row rank and there exists a nonsingular matrix $F_0(D) \in \Z_p^{k \times k}[D]$ such that $F_0(D)\overline{G(D)}=\left[ \begin{array}{c} \widetilde{G}_0(D) \\ 0 \end{array}
 \right]$ modulo $p$, where $\widetilde{G}_0(D)$ is full row rank with rank $k_0$. Considering ${F}_0(D) \in \Z_{p^r}^{k \times k}[D]$, it follows that
${F}_0(D){G}(D)=\left[\begin{array}{c}{G}_0(D) \\p \widehat{G}_1(D) \end{array}\right]$, where $G_0(D) \in \Z_{p^r}^{k_0 \times n}$ is such that $\overline{G_0(D)}=\widetilde{G}_0(D)$. Moreover, since $F_0(D)$ is invertible, $\left[\begin{array}{c}{G}_0(D) \\p \widehat{G}_1(D) \\\end{array}\right]$ is also a generator matrix of $\C$.

\noindent Let us now consider $F_1(D) \in \Z_{p}^{(k-k_0) \times (k-k_0)}[D]$ such that
$F_1(D)\overline{\widehat{G}_1(D)}=\left[\begin{array}{c}\widetilde{G}_1(D) \\ 0 \\ \end{array}\right]$ modulo $p$, where $\widetilde{G}_1(D)$ is full row rank with rank $k_1$.
Then, considering $F_1(D) \in \Z_{p^r}^{(k-k_0) \times (k-k_0)}[D]$, it follows that ${F}_1(D)\widehat{G}_1(D)=\left[\begin{array}{c} {G}'_1(D) \\ p \widehat{G}_2(D) \end{array} \right]$, where $G'_1(D) \in \Z_{p^r}^{\tilde k_1 \times n}$ is such that $\overline{G'_1(D)}=\widetilde G(D)$, and therefore
$$
\left[
  \begin{array}{cc}
    I_{k_0} & 0 \\
    0 & F_1(D)
  \end{array}
\right]F_0(D)G(D)=
\left[
  \begin{array}{c}
    G_0(D) \\
    pG'_1(D) \\
    p^2\widehat{G}_2(D)
  \end{array}
\right].
$$
If $\left[
  \begin{array}{c}
    G_0(D) \\
    G'_1(D)
  \end{array}
\right]$ is not full row rank, then there exists a permutation matrix $P$ and a rational matrix $L(D) \in \Z_{p^r}^{\tilde k_1 \times k_0}(D)$ such that
$$
P\left[
  \begin{array}{cc}
    I_{k_0} & 0 \\
    L_1(D) & I_{k_1}
  \end{array}
\right]\left[
  \begin{array}{c}
    G_0(D) \\
    pG'_1(D)
  \end{array}
\right]=
\left[
  \begin{array}{c}
    G_0(D) \\
    pG''_1(D) \\
    p^2G'_2(D)
  \end{array}
\right],
$$
where $G''_1(D) \in \Z_{p^r}^{k_1 \times n}(D)$ and $G'_2(D) \in \Z_{p^r}^{(\tilde k_1-k_1) \times n}(D)$ are rational matrices and $\left[
  \begin{array}{c}
    G_0(D) \\
    G''_1(D)
  \end{array}
\right]$ is a full row rank rational matrix. Note that since $P\left[
  \begin{array}{cc}
    I_{k_0} & 0 \\
    L_1(D) & I_{k_1}
  \end{array}
\right]$ is nonsingular it follows that
$$\mbox{Im}_{\Z_{p^r}((D))} \left[
  \begin{array}{c}
    G_0(D) \\
    pG'_1(D)
  \end{array}
\right] =  \mbox{Im}_{\Z_{p^r}((D))} \left[
  \begin{array}{c}
    G_0(D) \\
    pG''_1(D) \\
    p^2G'_2(D)
  \end{array}
\right].$$
Let $G_1(D) \Z_{p^r}^{k_1 \times n}[D]$ and $G''_2(D)\in \Z_{p^r}^{(\tilde k_1-k_1) \times n}[D]$ be polynomial matrices (see Lemma \ref{rational encoder}) such that $$\mbox{Im}_{\Z_{p^r}((D))} \left[
  \begin{array}{c}
    G_0(D) \\
    pG''_1(D) \\
    p^2G'_2(D)
  \end{array}
\right] =  \mbox{Im}_{\Z_{p^r}((D))} \left[
  \begin{array}{c}
    G_0(D) \\
    pG_1(D) \\
    p^2G''_2(D)
  \end{array}
\right].$$ Then $$
\left[
  \begin{array}{c}
    G_0(D) \\
    pG_1(D) \\
    p^2G''_2(D)\\
    p^2\widehat{G}_2(D)
  \end{array}
\right]$$
is still a generator matrix of $\C$ such that $ \left[ \begin{array}{c}
    G_0(D) \\
    G_1(D)
  \end{array}
\right] $ is full row rank.

\noindent Proceeding in the same way we obtain a generator matrix of $\cal C$ of the form $$
\left[
       \begin{array}{c}
         G_0(D) \\
         pG_1(D) \\
         \vdots \\
         p^{r-1}G_{r-1}(D) \\
       \end{array}
     \right],
$$
and such that
$$
\left[
       \begin{array}{c}
         G_0(D) \\
         G_1(D) \\
         \vdots \\
         G_{r-1}(D) \\
       \end{array}
     \right]
$$
is full row rank. Thus $\C_i:= \mbox{Im} \, G_i(D)$ is a free convolutional code, $i=0, 1, \dots, r-1$, and $\C=\C_0 \oplus p \C_1 \oplus \dots \oplus p^{r-1} \C_{r-1}$. If we denote by $k_i$ the rank of $\C_i$ then the family $\{k_0,\ldots,k_{r-1}\}$ is a characteristic of the code. Moreover, it's clear that $\C$ is free if and only if $k_i=0$ for $i=1\ldots r-1$.

\bigskip

\noindent The following lemmas will be very useful for deriving the results of the remaining sections.

\begin{lemma} \label{1_0}
Let $\C$ be a free convolutional code of length $n$ over $\Z_{p^r}((D))$ and rank $k$. Then, $p$-$dim(p^i\C)=(r-i)k$.
\end{lemma}

\begin{proof}
Let $G(D)\in \Z_{p^r}^{k \times n}[D]$ be an encoder of $\C$. The result follows from the fact that $\left[
       \begin{array}{c}
         p^iG(D) \\
         p^{i+1}G(D) \\
         \vdots \\
         p^{r-1}G(D) \\
       \end{array}
     \right]$ is an $p$-encoder of $\C$, since $G(D)$ is full row rank.
     \qed
\end{proof}

\begin{lemma}
\label{l_1}
Let $\mathcal{C}_1$ and $\mathcal{C}_2$ be two convolutional codes over $\ZZ((D))$. Then we have
\begin{equation*}
  p\mbox{-}dim(\mathcal{C}_1+\mathcal{C}_2)= p\mbox{-}dim\,\mathcal{C}_1+p\mbox{-}dim\,\mathcal{C}_2-p\mbox{-}dim(\mathcal{C}_1 \cap \mathcal{C}_2).
  \end{equation*}
If the sum is direct, we have
\begin{equation*}
  p\mbox{-}dim(\mathcal{C}_1 \oplus \mathcal{C}_2)= p\mbox{-}dim\,\mathcal{C}_1+p\mbox{-}dim\,\mathcal{C}_2.
  \end{equation*}
\end{lemma}
\begin{proof} Suppose that $\C_1$ and $\C_2$ are in direct sum, i.e, $\C_1+\C_2=\C_1\oplus\C_2$.\\ If $B_1$ is a $p$-basis of $\C_1$ and $B_2$ is a $p$-basis of $\C_2$, then $(B_1,B_2)$ is a $p$-basis of $\C_1\oplus\C_2$ which gives the result.\\
For the general case, Let denote by $A$  the complement of $\C_1\cap\C_2$ in $\C_1$, i.e., $\C_1=A\oplus \C_1\cap\C_2$, and let $B$ such that $\C_2=B\oplus\C_1\cap\C_2.$  Then we have $$\C_1+\C_2=A\oplus\C_1\cap\C_2\oplus B$$ and the result is immediate.
\qed
\end{proof}

\noindent Next corollary follows immediately from Lemmas \ref{1_0} and \ref{l_1}.

\begin{corollary}
Let $\C$ be a convolutional code of length $n$ such that
$$\C=\C_0 \oplus p \C_1 \oplus \dots \oplus p^{r-1} \C_{r-1}$$
with $\C_i$ a free convolutional code with rank $k_i$, $i=0, 1, \dots, r-1$. Then $$p\mbox{-}dim(\C)=\sum_{i=0}^{r-1}(r-i)k_i.$$
\end{corollary}


\section{Dual Codes}
\label{sec:3}
Let $\mathcal{C}$ be a convolutional code of length $n$ over $\ZZ((D))$. The {\bf orthogonal} of $\mathcal{C}$, denoted by $\mathcal{C}^{\perp}$, is defined as
\begin{equation*}
  \mathcal{C}^{\perp}=\{y \in \ZZ^n:[y,x]=0 \mbox{ for all } x \in \mathcal{C}\},
  \end{equation*}
where $[y,x]$ denotes the inner product over $\ZZ^n$.

\bigskip

\noindent In this section we will show that the dual of a convolutional code is still a convolutional code. The next theorem proves this statement for free convolutional codes.

\begin{theorem}\label{t_1}
Let $\C$ be a free convolutional code with length $n$ over $\Z_{p^r}((D))$ and rank  $\tilde k$. Then $\mathcal{C}^{\perp}$ is also a free convolutional code of length $n$ and rank  $n-\tilde k$.
\end{theorem}
\begin{proof}
Let $G(D) \in \Z_{p^r}^{ \tilde k \times n}$ be an encoder of $\C$. Since $G(D)$ is full row rank there exists a polynomial matrix  $L(D) \in \Z_{p^r}^{(n-\tilde k) \times n}[D]$ such that $\left[\begin{array}{c} G(D) \\ L(D) \end{array} \right]$ is nonsingular. Let $[X(D) \; Y(D)]$, with $X(D) \in \Z_{p^r}^{n \times \tilde k}(D)$ and $Y(D) \in \Z_{p^r}^{n \times (n- \tilde k)}(D)$, be the inverse of $\left[\begin{array}{c} G(D) \\ L(D) \end{array} \right]$. Then $\mathcal{C}^{\perp}=\mbox{Im}_{\Z_{p^r}((D))} Y(D)^t$, which means by Lemma \ref{rational encoder} that $\mathcal{C}^{\perp}$ is a convolutional code. Moreover, since $Y(D)$ is full column rank, there exists a full row rank polynomial matrix $G^{\perp}(D) \in  \Z_{p^r}^{(n- \tilde k) \times n}[D]$ such that $\C^{\perp}=\mbox{Im}_{\Z_{p^r}((D))} G^{\perp}(D)$. Thus $\C^{\perp}$ is a free convolutional code or rank $n-\tilde k$.
\qed
\end{proof}
If $\C$ is a free code of rank $\tilde k$, then $p\mbox{-}dim(\C)=\tilde{k}r$. This gives us the next corollary.
\begin{corollary}\label{c21}
Let $\C$ be a free convolutional code of length $n$ over $\Z_{p^r}$. Then we have $$p\mbox{-}dim(\C)+p\mbox{-}dim(\C^\bot)=nr.$$
\end{corollary}
\noindent In the sequel of this work we propose to established this result for any code over $\Z_{p^r}((D))$.\\

\noindent The following auxiliary lemmas will be fundamental in the proof of next theorem. They were proved in \cite{Mo15} for block codes over $\ZZ$ and are trivially extended for convolutional codes $\ZZ$((D)). We write their proofs for completeness.

\begin{lemma} \cite{Mo15}
Let $\C$ be a free convolutional code over $\Z_{p^r}((D))$. For any given integer $i \in \{0, \dots r-1\}$ we have
\begin{equation*}
\C \cap p^i \ZZ^n((D))=p^i \C.
\end{equation*}
\end{lemma}

\begin{proof} The inclusion $p^i \C\subset \C\cap p^i\ZZ^n((D))$ is trivial. Conversely, let $y(D) \in p^i\ZZ^n((D))\cap \C$.
Let $\{ x_1(D),\ldots,x_k(D) \}$ be a basis of $\C$ and its projection over $\Z_p((D))$ $\{ \overline{x_1(D)},\ldots,\overline{x_k(D)}\} $ be  a basis of $\overline{\C}$. Then, there exists $a_1(D),\ldots,a_k(D) \in \ZZ((D))$ such that $y(D)=\displaystyle\sum_{j=1}^ka_j(D) x_j(D)$. As $y(D) \in p^i\ZZ^n((D))$, we have $\overline{y(D)}=\displaystyle\sum_{j=1}^k\overline{a_j(D)}\overline{x_j(D)}=0$, and therefore $\overline{a_j(D)}=0$, $j=1\ldots k$. Then, for all $j=1\ldots k$, $a_j(D)$ can be written in the form $pb_j(D)$ where $b_j(D) \in \ZZ((D))$. By repeating the procedure $i$ times, we obtain $a_j(D)=p^i\alpha_j(D),\;\forall j=1\ldots k$, which gives  $$y(D)=p^i\sum_{j=1}^k\alpha_j(D)x_j(D)\in p^i\C.$$
\qed
\end{proof}

\begin{lemma}\label{l4}\cite{Mo15}
Suppose that $\C$ is a free code. Let $y(D) \in \ZZ((D))^n$  and let $i$ be an  integer in  $ \{0,\ldots, r-1\},$ such that $p^i y(D) \in \C$. Then $ y(D) \in \C+p^{r-i}\ZZ^n((D)).$
\end{lemma}

\begin{proof}
By the preceding lemma, there exists $x(D) \in\C$ such that $p^i y(D)=p^i x(D)$. This implies that $\overline{y(D)}=\overline{x(D)}$. Thus there exists $y_1(D) \in\C$, $y_2(D) \in\ZZ((D))$ satisfying  $y(D)=y_1(D)+p y_2(D).$ We have $p^i y(D)=p^i y_1(D)+p^{i+1}y_2(D)$, then $p^i y(D)-p^i y_1(D)=p^{i+1}y_2(D)\in \C$. Then $y_2(D)=y_3(D)+p y_4(D)$ where  $y_3(D)\in \C$ and $y_4(D)\in\ZZ^n((D))$. Then $y(D)=\underbrace{y_1(D)+py_3(D)}_{\in \C}+p^2y_4(D).$
By repeating this procedure $r-i$ times, we obtain $y(D)=x_1(D)+p^{r-i}x_2(D)$ with $x_1(D)\in \C.$
\qed
\end{proof}

\begin{lemma} \cite{Mo15}
Let $\C$ be a free convolutional code over $\Z_{p^r}((D))$. For all integer $i \in \{0, \dots r-1\}$ we have
\begin{equation*}
(p^i \C)^{\perp}=\C^{\perp} + p^{r-i} \ZZ^n((D)).
\end{equation*}
\end{lemma}

\begin{proof}
 It's clear that $\C^\bot+p^{r-i}\ZZ^n((D))\subset(p^i\C)^\bot$. Conversely, let $y(D) \in (p^i\C)^\bot$. Then for all $x(D) \in \C$ we have $[y(D),p^i x(D)]=[p^iy(D),x(D)]=0,$ thus $p^i y(D) \in \C^\bot$. As $\C^\bot$ is a free code, we conclude by  Lemma~\ref{l4} that $y(D) \in \C^\bot+p^{r-i}\ZZ^n((D))$.
 \qed
 \end{proof}

\begin{theorem}
Let $\C=\C_0\oplus p\C_1\oplus\ldots\oplus p^{r-1}\C_{r-1}$ be a convolutional code of length $n$ over $\Z_{p^r}((D))$, such that $\C_i$ is free, $i=0, 1, \dots, r-1$, with $\C_0\oplus \C_1\oplus\ldots\oplus \C_{r-1}= \C_0 + \C_1 + \ldots + \C_{r-1}$ a free convolutional code. Then, there exists a family of free convolutional codes of length $n$ over $\Z_{p^r}((D))$, $B_i,i=0,\ldots,r-1,$ such that
 $\C^\bot= B_0\oplus p B_1\oplus\ldots\oplus p^{r-1}B_{r-1},$  and
 \begin{enumerate}
 \item  $B_0=(\C_0\oplus\ldots\oplus \C_{r-1})^\bot.$
 \item For $i\in\{1,\ldots,r-1\},$ $rank(B_i)=rank(\C_{r-i})$.
 \end{enumerate}
\end{theorem}

\begin{proof}
Suppose that $rank(\C_i)=k_i$ for $i=0,\ldots,r-1$.  We first begin by looking for  the dual of  $\C_0\oplus p\C_1$.
 \begin{eqnarray*}
 (\C_0\oplus p\C_1)^\bot&=&\C_0^\bot\cap(p\C_1)^\bot=\C_0^\bot\cap(\C_1^\bot+p^{r-1}\Z_{p^r}^n)\\
 &=& \C_0^\bot\cap\C_1^\bot+p^{r-1}\C_0^\bot\\
 &=& (\C_0\oplus\C_1)^\bot +p^{r-1}\C_0^\bot.
 \end{eqnarray*}
\noindent  By Theorem 1, we can conclude  that there exists a free code $B_{r-1}$ such that $$(\C_0\oplus p\C_1)^\bot=(\C_0\oplus\C_1)^\bot \oplus p^{r-1}B_{r-1}.$$
  Suppose $rank(B_{r-1})=l_{r-1}$, then we have:
 \begin{eqnarray*}
 p\mbox{-}dim[(\C_0\oplus p\C_1)^\bot]&=&p\mbox{-}dim(\C_0\oplus\C_1)^\bot+ p\mbox{-}dim(p^{r-1}B_{r-1})\\
 &=&nr-(k_0+k_1)r+l_{r-1}.
 \end{eqnarray*}
\noindent
 On the other hand, $p\mbox{-}dim[(\C_0\oplus p\C_1)^\bot]=nr-(k_0r+(r-1)k_1)$. We conclude  that $rank(B_{r-1})=k_1$.
 We repeat the same procedure with $\C_0\oplus p\C_1\oplus p^2\C_2$.
 \begin{eqnarray*}
 (\C_0\oplus p\C_1\oplus p^2\C_2)^\bot&=&(\C_0\oplus p\C_1)^\bot\cap(p^2\C_2)^\bot=[(\C_0\oplus\C_1)^\bot\oplus p^{r-1}B_{r-1}]\cap(\C_2^\bot+p^{r-2}\Z_{p^r}^n)\\
 &=&(\C_0\oplus\C_1\oplus\C_2)^\bot\oplus p^{r-1}(B_{r-1}\cap\C_2^\bot)+p^{r-2}(\C_0\oplus\C_1)^\bot+p^{r-1}B_{r-1}\\
 &=&(\C_0\oplus\C_1\oplus\C_2)^\bot\oplus p^{r-1}B_{r-1}+p^{r-2}(\C_0\oplus \C_1)^\bot.
 \end{eqnarray*}
 By Theorem 1, there exists a free convolutional code $B_{r-2}$ such that $$(\C_0\oplus p\C_1\oplus p^2\C_2)^\bot=(\C_0\oplus\C_1\oplus\C_2)^\bot\oplus p^{r-1}B_{r-1}\oplus p^{r-2}B_{r-2}.$$
  Suppose that $rank(B_{r-2})=l_{r-2}$, then we have
 \begin{eqnarray*}
 p\mbox{-}dim(\C_0\oplus p\C_1\oplus p^2\C_2)^\bot&=&p\mbox{-}dim[(\C_0\oplus\C_1\oplus\C_2)^\bot]+p\mbox{-}dim(p^{r-1}B_{r-1})+p\mbox{-}dim(p^{r-2}B_{r-2})\\
 &=&nr-(k_0+k_1+k_2)r+k_1+2l_{r-2}.
 \end{eqnarray*}
 On the other hand
 \begin{eqnarray*}
 p\mbox{-}dim(\C_0\oplus p\C_1\oplus p^2\C_2)^\bot&=&nr-[k_0r+k_1(r-1)+k_2(r-2)]\\
 &=&(n-k_0-k_1)r+k_1+2k_2.
 \end{eqnarray*}
 We conclude that $rank(B_{r-2})=k_2$. Repeating this procedure $r-1$ times we obtain the desired result.
\qed
\end{proof}
%

\noindent The following result is a consequence of this theorem and generalizes the well-known result for the field case: if $\C$ is a convolutional code of length $n$ over $\mathbb F((D))$, where $\mathbb F$ is a finite field, then $dim \C + dim \C^{\bot} = dim \mathbb F((D))=n$.

\begin{corollary}
Let $\C$ be a convolutional code of length $n$ over $\Z^n_{p^r}$. Then
$$
p\mbox{-}dim(\C)+ p\mbox{-}dim(\C^{\bot}) = p\mbox{-}dim(\Z^n_{p^r}((D))=nr.
$$
\end{corollary}

\begin{proof}
Let $\C=\C_0\oplus p\C_1\oplus\ldots\oplus p^{r-1}\C_{r-1}$ where $\C_i$ is free of rank $k_i$, $i=0, 1, \dots, r-1$. Consider also the free convolutional codes of length $n$ over $\Z_{p^r}((D))$, $B_i,i=0,\ldots,r-1,$ such that
 $\C^\bot= B_0\oplus p B_1\oplus\ldots\oplus p^{r-1}B_{r-1}$, and
 \begin{enumerate}
 \item  $B_0=(\C_0\oplus\ldots\oplus \C_{r-1})^\bot$.
 \item  $rank(B_i)=rank(\C_{r-i})$, $i\in\{1,\ldots,r-1\}$.
 \end{enumerate}
  Note that $p\mbox{-}dim(\C)=\displaystyle\sum_{i=0}^{r-1}(r-i)k_i$. From 2. and Lemma \ref{1_0}, it follows that $p\mbox{-}dim(p^iB_i)=(r-i)k_{r-i}$ and from 1. and Corollary \ref{c21} it follows that $p\mbox{-}dim(B_0)=nr-r(k_0+k_1, \cdots + k_{r-1})$. Thus,
\begin{eqnarray*}
 p\mbox{-}dim(\C^{\bot}) & = & p\mbox{-}dim(B_0) + p\mbox{-}dim(p B_1) + \cdots + p\mbox{-}dim(p^{r-1}B_{r-1}) \\
 & = & nr-r(k_0+k_1+ \cdots + k_{r-1}) + (r-1)k_{r-1} + (r-2)k_{r-2} + \cdots + k_1 \\
 & = & nr - (k_0r+k_1(r-1)+\cdots + k_{r-1})\\
 & = & nr -  p\mbox{-}dim(\C).
 \end{eqnarray*}
\qed
\end{proof}

\begin{remark}
{\it In the case of block code over a finite ring, we can find this result using the theorem of J.Wood in \cite{wo}.
Indeed, if $\C$ is a block code of length $n$ over $\R$. $\R$ is a Frobenius ring and then we have $$|\C||\C^\bot|=|\R^n|.$$
If $p\mbox{-}dim(\C)=k$, we have $|\C|=p^k$ and then $|\C^\bot|=p^{nr-k}$ which gives} $$p\mbox{-}dim(\C^\bot)=nr-k.$$
\end{remark}

\begin{acknowledgement}
The work of the second, third and fourth authors was supported in part by the Portuguese Foundation for Science and Technology (FCT-Funda\c{c}\~ao para a Ci\^encia e a Tecnologia), through CIDMA - Center for Research and Development in Mathematics and Applications, within project UID/MAT/04106/2013.
\end{acknowledgement}

\bibliographystyle{spmpsci}
\bibliography{bibliog,biblio}

\end{document}